\documentclass[11pt]{article}
\usepackage{amssymb,amsmath,amsfonts}
\usepackage{stmaryrd}
\usepackage{color}
\usepackage[latin1]{inputenc}
\usepackage{multirow}
\usepackage{subfigure}
\usepackage[latin1]{inputenc}
 \usepackage{color}
\usepackage{graphicx}
\usepackage{url}
\usepackage{amsmath}
\usepackage{stmaryrd}
\usepackage{amsfonts}
\usepackage{epic}
\usepackage{pstricks}
\usepackage{pst-node}
\usepackage{pst-coil}
\usepackage{pstricks-add}
\usepackage{epstopdf}

\def\BBox{\kern  -0.2cm\hbox{\vrule width 0.2cm height 0.2cm}}

\newtheorem{lemma}{Lemma}[section]
\newtheorem{theorem}{Theorem}[section]
\newtheorem{definition}{Definition}[section]
\newtheorem{corollary}{Corollary}[section]

\newtheorem{remark}{Remark}[section]

\title{About $(k,l)$-kernels, semikernels and Grundy functions  in partial line digraphs}

\author{ C. Balbuena$^{1}$, H. Galeana-S\'anchez$^{2}$,
Mucuy-kak Guevara$^{3}$
\thanks{Research   supported by the Ministerio de Educación y Ciencia,
Spain, the European Regional Development Fund (ERDF) under project
MTM2014-60127-P. CONACyT-México under project 219840 and
PAPIIT-México under project IN108715. \newline
\footnotesize{\em
Email addresses:}
  ~ m.camino.balbuena@upc.edu (C.
Balbuena), \, \, ~ hgaleana@matem.unam.mx (H. Galeana-S\'anchez), \, \, ~ mucuy-kak.guevara@ciencias.unam.mx (M. Guevara)}
 \\[2ex]
 $^1${\footnotesize Departament de Matemática Aplicada III, Universitat
Politècnica de Catalunya, }\\
{\footnotesize Campus Nord, Edifici C2, C/ Jordi Girona 1 i 3 E-08034 Barcelona,
Spain.}\\
 $^2$ {\footnotesize Instituto de Matem\'{a}ticas, Universidad Nacional Aut\'onoma de M\'exico,} \\
{\footnotesize Ciudad Universitaria, Circuito Exterior, 04510 M\'exico D. F., M\'exico }\\
$^3$ {\footnotesize Facultad de Ciencias, Universidad Nacional Aut\'onoma de M\'exico,} \\
{\footnotesize Circuito Exterior S/N, Cd. Universitaria, Delegaci\'on Coyoac\'an, 04510, M\'exico D.F.}}

\date{}

\begin{document}
\maketitle
\begin{abstract}
Let $D=(V,A)$ be a digraph and consider an arc subset $A'\subseteq A$ and an exhaustive mapping $\phi: A\to A'$ such that
\begin{itemize}
\item[(i)] the set of heads of $A'$ is $H(A')=V$;
\item[(ii)] the map fixes the elements of $A'$, that is, $\phi|A'=Id$, and for every vertex $j\in V$, $\phi(\omega^-(j))\subset \omega^-(j)\cap A'$.
\end{itemize}
 Then, \emph{the partial line digraph} of $D$,
denoted by $\mathcal{L}_{(A',\phi)}D
 $ (for short $\mathcal{L}D$ if
the pair $(A', \phi)$ is clear from the context), is the digraph
with vertex set $V (\mathcal{L}D)=A'$  and set of arcs $A(\mathcal{L}D) = \{(ij, \phi(j,k)) : (j,k)\in A\}.$ In this paper we prove the following results:

Let $k,l$ be two natural numbers such that $1\le l \le k$, and $D$ a digraph with minimum in-degree at least 1. Then the number of $(k,l)$-kernels of $D$ is less than or equal to the number of $(k,l)$-kernels of $\mathcal{L} D$. Moreover, if $l<k$ and the girth of $D$ is at least $l+1$, then these two numbers are equal.

The number of semikernels of $D$ is equal to the number of semikernels of $\mathcal{L} D$.

Also we introduce the concept of $(k,l)$-Grundy function as a generalization of the concept of Grundy function and we prove that the number of $(k,l)$-Grundy functions of $D$ is equal to the number of $(k,l)$-Grundy functions of any partial line  digraph  $\mathcal{L} D$.

\end{abstract}

{\bf Keywords:}
Digraphs, in-domination, kernel, Grundy function.

{\bf MSC2010:} 05C20,  05C63


\section{Introduction}

 Throughout the paper, $D = (V ,A)$ denotes a loopless digraph with set of vertices $ V  $ and arc set  $A$.
 Let  $\omega^-(x)$  stand for the set of arcs having  vertex $x$ as their   terminal  vertex, and $\omega^+(x)$  stand for the set of arcs having  vertex $x$ as their  initial  vertex. Thus, the
  \emph{in-degree} of $x$ is $d^-(x) = |\omega^-(x)|$ and the \emph{out-degree} of $x$ is $d^+(x) = |\omega^+(x)|$. The \emph{minimum in-degree} (\emph{minimum out-degree}) of $D$ is $\delta^-(D)=\min \{d^-_D(x) : x\in V\}$ ($\delta^+(D)=\min \{d^+_D(x) : x\in V\}$ respectively).
    Moreover, given a set $U\subseteq V$, $\omega^-(U)=\{(x,y)\in A: y\in U \hbox{ and } x\notin U\}$.     Given a set of arcs $\Omega\subseteq A$,   the \emph{heads} of $\Omega$ are  the vertices in the set  $H(\Omega)=\{y: (x,y)\in \Omega\}$.  For any
pair of vertices $x,y\in V$, a directed path $(x,x_1,\ldots,x_{n-1},y)$ from $x$ to $y$   is called an
$x\rightarrow y$ path.   The \emph{distance} from $x$ to $y$ is denoted by $d_D(x,y)$ and it is defined
to be the length of a shortest $x\rightarrow y$ path.

A set $K\subset V(D)$ is said to be a \emph{kernel} if it is both independent (for every two vertices $x,y\in K$, $d_D(x,y)\ge 2$,) and absorbing (a vertex not in $K$ has a successor in $K$).  This concept was first introduced in \cite{NM:44}  by Von Neumann and Morgensten in the context of Game Theory as a solution for cooperative $n$-player games. The concept of a kernel is important to the theory of digraphs because it arises naturally in applications such as Nim-type games, logic, and facility location, to name a few. Several authors have been investigating sufficient conditions for the existence of kernels in digraphs, for a comprehensive survey see for example \cite{BG:03}  and \cite{F:09}. Also see Chapter 15 of \cite{HHS:98} for a summary.

Let $l,k$ be two integers such that $l\ge 1$ and $k\ge 2$.
  A  \emph{$(k,l)$-kernel of a digraph $D$ } is a subset of
vertices $K$ which is both \emph{$k$-independent} ($d_D(u,v)\ge k$ for all $u,v\in K$) and
\emph{$l$-absorbing} ($d_D(x,K)\le l $ for all $x\in V\setminus K$). Observe that any kernel is a
 $(2,1)$-kernel and  a quasikernel,  introduced in \cite{GPR:91},  is a $(2,2)$-kernel.
The concept of $(k,l)$-kernel is a nice, wide generalization of the concept of kernel; $(k,l)$-kernels  have been deeply studied by several authors, see for example
\cite{GH:14,SWW:07,SWW:08,WW:09}.

Grundy functions are very useful in the context of game theory and they are nearly related to kernels as a digraph with Grundy function has also a kernel.  Also the concept of semikernel is very close to that of kernel, because a digraph such that every induced subdigraph  has a nomempty semikernel has  a kernel.

The line digraph technique is a good general method for obtaining
large digraphs with fixed degree and diameter. In the line digraph
$L(D)$ of a digraph $D$, each vertex  represents an arc of $D$.
Thus, $V(L(D))=\{ uv : (u,v) \in A(D)\} $; and  a vertex $uv$ is
adjacent to a vertex $xz$ if and only if $v=x$, that is, when  the
arc $(u,v)$ is adjacent to the arc $(x,z)$ in $D$. For any $h >
1$,  the \emph{$h$-iterated line digraph}, $L^h(D)$, is defined
recursively by  $L^h(D)=L(L^{h-1}(D))$.  For more information about line digraphs
 see,  for instance, Aigner \cite{A:67}, Fiol, Yebra
and Alegre \cite{FYA:84} and Reddy, Kuhl, Hosseini and Lee
\cite{RKHL:82}.

A wider family of digraphs, called partial line digraphs, was introduced in \cite{FLl92} as a generalization of line digraphs.
Let $D = (V ,A)$ be a digraph  and consider an arc subset
$A'\subseteq A$ and   an exhaustive mapping $\phi : A \to A'$ such
that:
 \begin{itemize}
  \item[(i)] the set of heads of $A'$ is $H(A')=V$;
  \item[(ii)] the map $\phi$ fixes the elements of $A'$, that is, $\phi|A' = id$, and  for every
  vertex $j\in V$,  $\phi(\omega^-(j))\subset
\omega^-(j)\cap A'$.
\end{itemize}
Hence, $|V| \le |A'|\le |A|$. Note that the existence of such a
subset $A'$ is guaranteed since $\delta^-(i)\ge 1$ for every $i\in V$.
 Then, \emph{the partial line digraph} of $D$,
denoted by $\mathcal{L}_{(A',\phi)}D
 $ (for short $\mathcal{L}D$ if
the pair $(A', \phi)$ is clear from the context), is the digraph
with vertex set $V (\mathcal{L}D)=A'$  and set of arcs
$$A(\mathcal{L}D) = \{(ij, \phi(j,k)) : (j,k)\in A\}.$$

\begin{remark}\label{uno}
If $A' = A$, then $\phi = id$ and
the partial line digraph $\mathcal{L}D$ coincides with the line
digraph $L(D)$.
\end{remark}

Fig. \ref{FigLP1} shows an example of a digraph $D$ with  12 arcs and its partial line digraph with
$|A'|= 9$ vertices. The arcs not in $A'$ are drawn with dotted lines and have images $\phi(12) = 42$,
$\phi(34) = 54$, and $\phi(65) = 25$.

\begin{figure}[h!]
\begin{center}
\psset{unit=0.9cm, linewidth=0.035cm}
   \begin{pspicture}(-1,-1)(12,5.5)
      \cnode*(1,5){0.1}{1} \put(1.3,5){\footnotesize \it $1$}
      \cnode*(4,-.5){0.1}{3} \put(4.3,-0.6){\footnotesize \it $3$}
      \cnode*(-2,-.5){0.1}{6} \put(-2.5,-0.6){\footnotesize \it $6$}
      \cnode*(1,2.5){0.1}{2} \put(1.3,2.6){\footnotesize \it $2$}
      \cnode*(2,.75){0.1}{4} \put(2.2,0.8){\footnotesize \it $4$}
      \cnode*(0,0.75){0.1}{5} \put(-.4,0.8){\footnotesize \it $5$}

      \put(4,3){\footnotesize   $\phi(12)=42$}
      \put(4,2){\footnotesize   $\phi(34)=54$}
      \put(4,1){\footnotesize   $\phi(65)=25$}

      \ncline[arrowsize=.2 3]{->}{1}{3}
      \ncline[arrowsize=.2 3]{->}{3}{6}
      \ncline[arrowsize=.2 3]{->}{6}{1}
      \ncline[arrowsize=.2 3]{->}{2}{5}
      \ncline[arrowsize=.2 3]{->}{5}{4}
      \ncline[arrowsize=.2 3]{->}{4}{2}

     \ncarc[arcangle=-12, arrowsize=.2 2]{->}{2}{1}
      \ncarc[arcangle=-12, arrowsize=.2 2, linestyle=dashed, linecolor=blue]{->}{1}{2}
      \ncarc[arcangle=-12, arrowsize=.2 2, linestyle=dashed, linecolor=blue]{->}{3}{4}
      \ncarc[arcangle=-12, arrowsize=.2 2]{->}{4}{3}
      \ncarc[arcangle=-12, arrowsize=.2 2]{->}{5}{6}
      \ncarc[arcangle=-12, arrowsize=.2 2, linestyle=dashed, linecolor=blue]{->}{6}{5}


      \cnode*(11,4.3){0.1}{21} \put(10.8,4.5){\footnotesize \it $21$}
       \cnode*(7,4.3){0.1}{56} \put(6.6,3.9){\footnotesize \it $56$}
      \cnode*(9,.5){0.1}{43} \put(9.2,.5){\footnotesize \it $43$}

      \cnode*(8,3.7){0.1}{25}\put(7.5,3.4){\footnotesize \it $25$}
      \cnode*(9,1.7){0.1}{54} \put(9.25,1.5){\footnotesize \it $54$}
      \cnode*(10,3.7){0.1}{42}\put(9.7,4){\footnotesize \it $42$}

      \cnode*(9,-0.5){0.1}{36} \put(8.8,-1){\footnotesize \it $36$}
      \cnode*(12,5){0.1}{13} \put(12.2,5){\footnotesize \it $13$}
      \cnode*(6,5){0.1}{61} \put(5.5,5){\footnotesize \it $61$}

      \ncline[arrowsize=.2 2]{->}{25}{54}
      \ncline[arrowsize=.2 2]{->}{54}{42}
      \ncline[arrowsize=.2 2]{->}{42}{25}
      \ncline[arrowsize=.2 2]{->}{13}{36}
      \ncline[arrowsize=.2 2]{->}{36}{61}
      \ncline[arrowsize=.2 2]{->}{61}{13}

      \ncline[arrowsize=.2 2]{->}{21}{13}
      \ncline[arrowsize=.2 2]{->}{56}{61}
      \ncline[arrowsize=.2 2]{->}{43}{36}

      \ncarc[arcangle=29, arrowsize=.2 2]{->}{25}{56}
      \ncarc[arcangle=-29, arrowsize=.2 2]{->}{42}{21}
      \ncarc[arcangle=29, arrowsize=.2 2]{->}{54}{43}
      \ncarc[arcangle=-29, arrowsize=.2 2, linecolor=blue]{->}{21}{42}
      \ncarc[arcangle=29, arrowsize=.2 2, linecolor=blue]{->}{56}{25}
      \ncarc[arcangle=29, arrowsize=.2 2, linecolor=blue]{->}{43}{54}

      \ncline[arrowsize=.2 2, linecolor=blue]{->}{61}{42}
      \ncline[arrowsize=.2 2, linecolor=blue]{->}{36}{25}
      \ncline[arrowsize=.2 2, linecolor=blue]{->}{13}{54}

      \end{pspicture}
\caption{A digraph and its partial line digraph. \label{FigLP1}}
\end{center}
\end{figure}
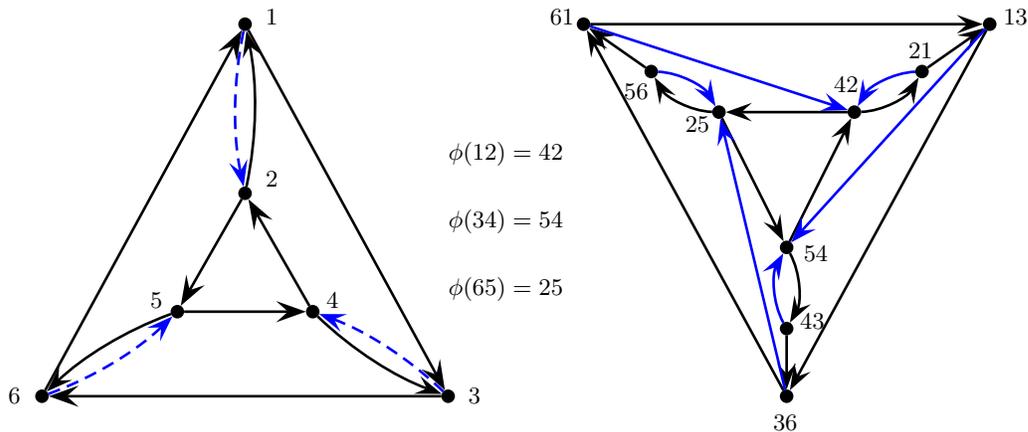

In this paper we study the relationship between the number of $(k,l)$-kernels (resp. semikernels) of a digraph $D$ and the corresponding number in any partial line digraph $\mathcal{L} D$. Also we introduce the concept of $(k,l)$-Grundy function as a generalization of the concept of Grundy function and we prove that the number of $(k,l)$-Grundy functions of $D$ is equal to the number of $(k,l)$-Grundy functions of any partial line  digraph  $\mathcal{L} D$.

\section{$(k,l)$-kernels and semikernels}

 In this section we will prove that the number of $(k,l)$-kernels of a digraph is less than or equal to the number of $(k,l)$-kernels of its partial line digraphs,  and under certain conditions these two numbers are equal.

We start by proving a result concerning independent sets of a digraph and of those of their partial line digraphs.

\begin{lemma}\label{kinde}
Let $D$ be a digraph with minimum in-degree at least 1. Let $A'$ and $\phi$ satisfy the requirements of the definition of a partial line digraph, i.e., $\mathcal{L}_{(A',\phi)}D=\mathcal{L} D$. Let $k\ge2$ be an integer number.  Denote by $\mathcal{I}$ the set of all $k$-independent sets of $D$, and by $\mathcal{I}^*$ the set of all $k$-independent set of  $\mathcal{L}D$. Then
the assignment $f:\mathcal{I} \to \mathcal{I}^*$  defined by $f(I)=\omega^-(I)\cap A'$ for all $I\in \mathcal{I}$ is an injective function.
Therefore the   number of $k$-independent sets of $D$ is less than or equal to  number  of  $k$-independent sets of $\mathcal{L} D$.
\end{lemma}
\begin{proof}
First of all let us see that $f$ is a function.
Let
$ab,cd\in \omega^-(I)\cap A'$ be such that
$d_{\mathcal{L} D}(ab,cd)=t$, and observe that
  $d_D(b,d)\ge k$   because     $b,d\in I$. By definition
   of $\phi$    any shortest path   from $ab$ to $cd$ in $\mathcal{L} D$ is
  $ab,\phi(bb_1), \phi(b_1b_2), \ldots, \phi(b_{t-1}b_t)=cd$,
   where $b_i\in V( D)$ and $(b,b_1), (b_i,b_{i+1})\in A(D)$, $i=1,\ldots, t-1$.
    Since $\phi(b_{t-1}b_t)=\alpha b_t$ for some $\alpha \in V(D)$,
    then $b_t=d$ yielding that a walk
$b,b_1,\ldots, b_t=d$ from $b$ to $d$ of length $t$ exists in $D$.
This means that $t\ge d_D(b,d)\ge  k$ and hence every two vertices
of $  \omega^-(I)\cap A'$ are mutually at distance at least $k$.

Let us prove that $f$ is an injective function. Let $I_1,I_2\in \mathcal{I}$ be such that  $f(I_1)=f(I_2)$, that is  $\omega^-(I_1)\cap A'=\omega^-(I_2)\cap A'$. Let us show that $I_1=I_2$. Let $u\in I_1$. Note that by item (i) of definition of $\mathcal{L} D$ there is $y\in V(D)$ such that
$yu\in A'$. Clearly, $yu\in \omega^-(I_1)\cap A'$ which implies that $yu\in \omega^-(I_2)\cap A'$, then $u\in I_2$, that is, $I_1\subseteq I_2$. Reasoning analogously, $I_2\subseteq I_1$ yielding that $I_1=I_2$. Therefore $f$ is an injective function and the lemma holds.
\end{proof}

 The concept of \emph{Fibonacci number}  for a graph $G$ was introduced in \cite{PT:82} and it is defined as the number of independent subsets of $G$ including the empty set. We extend this concept for digraphs, and  we give an upper bound on the Fibonacci number  of a digraph in terms of the Fibonacci number of its partial line digraph.
\begin{corollary}\label{fibo}
Let $D$ be a digraph with minimum in-degree at least 1. Let $A'$ and $\phi$ satisfy the requirements of the definition of a partial line digraph, i.e., $\mathcal{L}_{(A',\phi)}D=\mathcal{L} D$. Then the  Fibonacci number of  $D$ is less than or equal to the Fibonacci number  of   $\mathcal{L} D$.
\end{corollary}

\subsection{$(k,l)$-kernels}

Some known results about the existence of kernels and $(k,l)$-kernels
in line digraphs can be seen in  \cite{H:82,QEM:06}.
 The following theorem is proved in \cite{BG:10}.

\begin{theorem}\cite{BG:10}\label{BG:10}
Let $k,l$ be two natural numbers such that $1\leq l < k$, and let $D$ be a digraph with minimum in-degree at least 1 and girth at least $l+1$. Then $D$ has a $(k,l)$-kernel if and only if any partial line digraph $\mathcal{L} D$ has a $(k,l)$-kernel.
\end{theorem}

Note that, since a kernel is a $(2,1)$-kernel, it follows that  $D$ has a kernel if and only if any partial line digraph $\mathcal{L} D$ has a kernel.
Next,  we establish a relationship between the number of $(k,l)$-kernels of $D$ and the number of $(k,l)$-kernels of $\mathcal{L} D$.

\begin{theorem}\label{klker}
Let $k,l$ be two natural numbers such that $ l \ge 1$ and $ k\ge 2$, and let $D$ be a digraph with minimum in-degree at least 1. Let $A'$ and $\phi$ satisfy the requirements of the definition of a partial line digraph, i.e., $\mathcal{L}_{(A',\phi)}D=\mathcal{L} D$. Then
 the number of $(k,l)$-kernels of  $D$ is less than or equal to the number of $(k,l)$-kernels of    $\mathcal{L} D$.
Moreover, if $ l < k$ and the
  girth of $D$ is at least $l+1$, then these numbers are equal.
\end{theorem}
\begin{proof}
Denote by $\mathcal{K}$ the set of all $(k,l)$-kernels of $D$, and by $\mathcal{K}^*$ denote the set of all $(k,l)$-kernels of  $\mathcal{L}D$.

Let $f:\mathcal{K} \to \mathcal{K}^*$ be defined by $f(K)=\omega^-(K)\cap A'$ for all $K\in \mathcal{K}$. In the proof of Theorem 2.1 of \cite{BG:10} it was proved that $f$ is well defined. And from Lemma \ref{kinde} it follows that  $f$ is injective. Therefore $|\mathcal{K}|\le | \mathcal{K}^*|$.

Let $h:\mathcal{K}^* \to \mathcal{K}$ be defined by $h(\hat{K})=H(\hat{K})$ for all $\hat{K}\in \mathcal{K}^*$. In the proof of Theorem 2.1 of \cite{BG:10} it was proved that $h$ is well defined if  $ l < k$ and the girth of $D$ is at least $l+1$. Moreover, we can check that $h=f^{-1}$ because $h(f(K))=H(\omega^-(K)\cap A')=K$. Therefore the theorem holds.
\end{proof}

\begin{corollary}
Let $D$ be a digraph with minimum in-degree at least 1. Let $A'$ and $\phi$ satisfy the requirements of the definition of a partial line digraph, i.e., $\mathcal{L}_{(A',\phi)}D=\mathcal{L} D$. Then the following assertions hold:
\begin{itemize}
\item[(i)] The   number of kernels of $D$ is equal to the  number of kernels of   $\mathcal{L} D$.
\item[(ii)] The   number of quasikernels of $D$ is less than or equal to the  number of quasikernels of   $\mathcal{L} D$.
\end{itemize}

\end{corollary}

\begin{figure}[h!]
\begin{center}
\psset{unit=0.9cm, linewidth=0.035cm}
   \begin{pspicture}(-1,-1)(12,3.)

      \cnode*(1,1.75){0.1}{2} \put(1.1,1.85){\footnotesize \it $x$}
      \cnode*(2,0){0.1}{4} \put(2.2,0.1){\footnotesize \it $z$}
      \cnode*(0,0.){0.1}{5} \put(-.4,0.1){\footnotesize \it $y$}
      \cnode*(3,1.75){0.1}{3} \put(3.1,1.85){\footnotesize \it $t$}

      \ncline[arrowsize=.2 3]{->}{2}{3}
      \ncline[arrowsize=.2 3]{->}{3}{4}
      \ncline[arrowsize=.2 3]{->}{2}{5}
      \ncline[arrowsize=.2 3]{->}{5}{4}
      \ncline[arrowsize=.2 3]{->}{4}{2}

 \cnode*(7,1.75){0.1}{l2} \put(7.1,1.85){\footnotesize \it $xy$}
      \cnode*(8,0.){0.1}{l4} \put(8.,-0.4){\footnotesize \it $zx$}
      \cnode*(6,0.){0.1}{l5} \put(5.6,-0.4){\footnotesize \it $yz$}
       \cnode*(10,0.){0.1}{l1} \put(10.,-0.4){\footnotesize \it $tz$}
  \cnode*(9,1.75){0.1}{l3} \put(9.1,1.85){\footnotesize \it $xt$}

      \ncline[arrowsize=.2 3]{->}{l2}{l5}
      \ncline[arrowsize=.2 3]{->}{l5}{l4}
      \ncline[arrowsize=.2 3]{->}{l4}{l2}
       \ncline[arrowsize=.2 3]{->}{l1}{l4}
   \ncline[arrowsize=.2 3]{->}{l3}{l1}
       \ncline[arrowsize=.2 3]{->}{l4}{l3}
      \end{pspicture}
\caption{A digraph with 3 quasikernels and its  line digraph with 5. \label{contra}}
\end{center}
\end{figure}
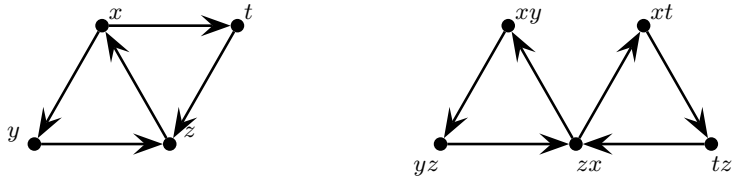
Let us observe that  the number of quasikernels of $D$ can be strictly less than  the  number of quasikernels of  its line digraph $L(D)$. A quasikernel is a $(2,2)$-kernel as we mentioned before, i.e., $k=l=2$. The digraph shown in Figure \ref{contra} proves that  the hypothesis $k<l$ can no be avoided to guarantee that the number of quasikernels of $D$ and $L(D)$ is equal. In this example the digraph $D$ on the left side hast 3 quasikernels, namely, $\{x\}$, $\{z\}$ and $\{y,t\}$, while its line digraph on the right side has 5 quasikernels which are $\{zx\}$, $\{tz,yz\}$, $\{xy,xt\}$, $\{xt,yz\}$ and $\{xy,tz\}$.

\subsection{Semikernels}
Let $S$ be an independent set of $D$. We say that $S$ is a \emph{semikernel} of $D$ if for all  $sx\in \omega^+(S)$  there exists $xs'\in \omega^-(S)$.
Thus, a vertex of out-degree zero forms a semikernel. Also a vertex only incident with symmetric arcs forms a semikernel. Figure \ref{pentas} depicts a digraph having a semikernel but not a kernel.
 In \cite{N:71} it was proved that if every induced subdigraph of $D$ has a (nonempty) semikernel, then every induced subdigraph of $D$ has a kernel, and so $D$.

\begin{figure}[h!]
\begin{center}
\psset{unit=0.9cm, linewidth=0.035cm}
   \begin{pspicture}(-4,-0.5)(9,2.6)

      \cnode*(1.5,-0.5){0.1}{1}
      \cnode*(0,0.7){0.1}{2}
      \cnode*(1,0.7){0.1}{3}
      \cnode*(2,0.7){0.1}{4}
      \cnode*(3,0.7){0.1}{5}
      \cnode*(0,2.){0.1}{6}
      \cnode*(1,2.){0.1}{7}
       \cnode*(2,2.){0.1}{8}
  \cnode*(3,2){0.1}{9} \put(3.,2.2){\footnotesize \it $x$}

      \ncline[arrowsize=.2 3]{->}{2}{1}
      \ncline[arrowsize=.2 3]{->}{1}{3}
      \ncline[arrowsize=.2 3]{->}{3}{7}
      \ncline[arrowsize=.2 3]{->}{7}{6}
      \ncline[arrowsize=.2 3]{->}{6}{2}

 \ncline[arrowsize=.2 3]{->}{4}{1}
      \ncline[arrowsize=.2 3]{->}{1}{5}
      \ncline[arrowsize=.2 3]{->}{5}{9}
      \ncline[arrowsize=.2 3]{->}{9}{8}
      \ncline[arrowsize=.2 3]{->}{8}{4}
\ncarc[arcangle=30, arrowsize=.2 2]{->}{8}{9}

      \end{pspicture}
\caption{A digraph with semikernel $\{x\}$ but without kernels. \label{pentas}}
\end{center}
\end{figure}
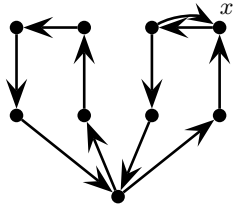
In Theorem 2.1 of \cite{GPR:91} it was proved that the number of semikernels of a digraph   with minimum in-degree at least one  is less than or equal to the number of semikernels of  the line digraph. Next we improve and generalize this result by stating the equality for every partial line digraph.

\begin{theorem}\label{klker}
 Let $D$ be a digraph with minimum in-degree at least 1. Let $A'$ and $\phi$ satisfy the requirements of the definition of a partial line digraph, i.e., $\mathcal{L}_{(A',\phi)}D=\mathcal{L} D$. Then the number of semikernels of  $D$ is  less than or equal to the number of semikernels of    $\mathcal{L} D$.
\end{theorem}

  \begin{proof} Denote by $\mathcal{S}$ the set of all semikernels of $D$, and by $\mathcal{S}^*$ denote the set of all semikernels of  $\mathcal{L}D$.
Let $f:\mathcal{S} \to \mathcal{S}^*$ be defined by $f(K)=\omega^-(K)\cap A'$ for all $K\in \mathcal{S}$.
 Let us see that $f(K)\in  \mathcal{S}^*$.

 By Lemma \ref{kinde}, we know $f(K)$ is an independent set. Let  $e'e\in \omega^+(f(K))$.
  Then $e'=x'y'\in f(K)= \omega^-(K)\cap A'$, yielding that
  $y'\in K$. Moreover,
  $e=\phi(y'y)$ because $e'e\in A(\mathcal{L} D)$,  which implies that  $y'y\in \omega^+(K)$ since $y'\in K$.
 Since $K$ is a semikernel, there exists  $yy''\in \omega^-(K)$, implying $e''=\phi(yy'')\in\omega^-(K)\cap A'=f(K)$, then $ee''\in \omega^-(f(K))$, implying that  $f(K)$ is a semikernel.
\end{proof}

  Figure \ref{semikernels} shows both a digraph and its line digraph with different number of semikernels.

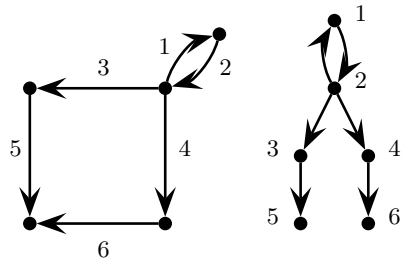
\begin{figure}[h!]
\begin{center}
\psset{unit=0.9cm, linewidth=0.035cm}
   \begin{pspicture}(-4,-0.5)(9,3.2)

      \cnode*(1,2){0.1}{1}
      \cnode*(-1,2){0.1}{2}
      \cnode*(-1,0){0.1}{3}
      \cnode*(1,0){0.1}{4}
      \cnode*(1.8,2.8){0.1}{5}

      \cnode*(3.5,3){0.1}{6} \put(3.8,3){\footnotesize $1$}
      \cnode*(3.5,2){0.1}{7} \put(3.8,2){\footnotesize $2$}
       \cnode*(3,1){0.1}{8} \put(2.5,1){\footnotesize $3$}
  \cnode*(4,1){0.1}{9} \put(4.3,1){\footnotesize $4$}
 \cnode*(3,0){0.1}{10} \put(2.5,0){\footnotesize $5$}
  \cnode*(4,0){0.1}{11} \put(4.3,0){\footnotesize $6$}

      \ncline[arrowsize=.2 3]{->}{1}{2} \put(0,2.2){\footnotesize $3$}
      \ncline[arrowsize=.2 3]{->}{1}{4} \put(1.2,1){\footnotesize $4$}
      \ncline[arrowsize=.2 3]{->}{2}{3} \put(-1.3,1){\footnotesize $5$}
      \ncline[arrowsize=.2 3]{->}{4}{3} \put(0,-.5){\footnotesize $6$}
      \ncarc[arcangle=30, arrowsize=.2 3]{->}{1}{5} \put(.9,2.5){\footnotesize $1$}
      \ncarc[arcangle=30, arrowsize=.2 3]{->}{5}{1} \put(1.8,2.2){\footnotesize $2$}

 \ncarc[arcangle=30, arrowsize=.2 3]{->}{6}{7}
      \ncarc[arcangle=30, arrowsize=.2 3]{->}{7}{6}
 \ncline[arrowsize=.2 3]{->}{7}{8}
   \ncline[arrowsize=.2 3]{->}{7}{9}
      \ncline[arrowsize=.2 3]{->}{8}{10}
      \ncline[arrowsize=.2 3]{->}{9}{11}

      \end{pspicture}
\caption{A digraph with $3$ semikernels and its line digraph, with $6$ semikernels. \label{semikernels}}
\end{center}
\end{figure}

\begin{theorem}
Let $D$ be a digraph with minimum in-degree at least 1. Let $A'$ and $\phi$ satisfy the requirements of the definition of a partial line digraph, i.e., $\mathcal{L}_{(A',\phi)}D=\mathcal{L} D$. Then $\mathcal{L} D$ has a semikernel if and only if $D$ has a semikernel.
\end{theorem}

\begin{proof} From Theorem \ref{klker} it follows that if $D$ has a semikernel, then $\mathcal{L} D$ has a semikernel. To see the converse let us consider the function
$h: \mathcal{S}^* \to \mathcal{S}$ defined by $h(K^*)=H(K^*)$. First, let us see that $H(K^*)$ is an independent set of $D$. Let $u,v\in H(K^*)$. Then $u'u,v'v\in K^*$ for some $u',v'\in V(D)$, yielding that $u'u,v'v$ are independent in $\mathcal{L} D$. We reason by contradiction assuming that $uv\in A(D)$. Then $(u'u,\phi(uv))\in \omega^+(K^*)$.
Since $K^*$ is a semikernel
 there is  $(\phi(uv),\phi(vw))\in \omega^-(K^*)$. But $v'v,\phi(vw)\in K^*$ and they are adjacent which is a contradiction. Therefore $u,v$ are not adjacent and so $H(K^*)$ is independent.

 Second, $vu\in \omega^+(H(K^*))$.  As $v\in H(K^*)$,  there is $v'v\in K^*$ and $(v'v,\phi(vu))\in  \omega^+(K^*)$. Since $K^*$ is a semikernel it follows that there exists  $(\phi(vu),\phi(uw))\in \omega^-(K^*)$. Then $\phi(uw)\in K^*$ and so $w\in H(K^*)$.
 Therefore $uw\in \omega^-(H(K^*))$ and the proof is finished.
\end{proof}

\section{Grundy function}

\begin{definition} \label{grun}
Consider a simple digraph  $D=(V,A)$. Following Berge \cite{B:85}, a non-negative integer function $g$ on $V$ is defined as a \emph{Grundy function}     if the following two requirements hold:
\begin{itemize}
\item[(1)] $g(x)=k>0$ implies that for each $j<k$, there is  $y\in N^+(x)$ with $g(y)=j$;
\item[(2)] $g(x)=k$ implies that each $y\in N^+(x)$ satisfies $g(y)\ne k$.
\end{itemize}
\end{definition}
This concept was first defined by P.M. Grundy in 1939 for acyclic digraphs as follows:

For every $x\in V$,  $g(x)=\min (\mathbb{N}\setminus \{g(y): y\in N^+(x)\})$.

\noindent Furthermore, Grundy proved that  every acyclic digraph has a unique Grundy function. However, there are digraphs without Grundy function, for instance the odd directed cycles. One of the most relevant properties of a Grundy function is that if $D$ has a Grundy function $g$, then $D$ has a kernel $K=\{ x\in V : g(x)=0\}$.

Next, we propose a generalization of a Grundy function called $(k,l)$-Grundy function. To do that we need to introduce some notation. The  out-neighborhood at distance $r$ from a vertex $x\in V$ is $N^+_r(x)=\{ y\in V: 1\le d(x,y)\le r\}$.
\begin{definition} \label{grunkl}
Consider a simple digraph  $D=(V,A)$ and let $l\ge 1$ and $k\ge 2$ be two integers. A non-negative integer function $g$ on $V$ is defined as a \emph{$(k,l)$-Grundy function}     if the following two requirements hold:
\begin{itemize}
\item[(1)] $g(x)=t>0$ implies that for each $j<t$, there is  $y\in N^+_l(x)$ with $g(y)=j$;
\item[(2)] $g(x)=t$ implies that each $y\in N^+_{k-1}(x)$ satisfies $g(y)\ne t$.
\end{itemize}
\end{definition}

Figure \ref{egrun} depicts on the left side a digraph with a $(2,2)$-Grundy function and on the right side a digraph with a $(3,2)$-Grundy function.
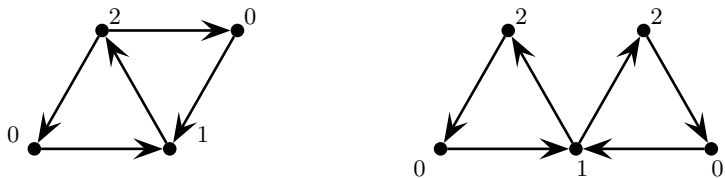
\begin{figure}[h!]
\begin{center}
\psset{unit=0.9cm, linewidth=0.035cm}
   \begin{pspicture}(-1,-1)(12,3.)

      \cnode*(1,1.75){0.1}{2} \put(1.1,1.85){\footnotesize \it $2$}
      \cnode*(2,0){0.1}{4} \put(2.4,0.1){\footnotesize \it $1$}
      \cnode*(0,0.){0.1}{5} \put(-.4,0.1){\footnotesize \it $0$}
      \cnode*(3,1.75){0.1}{3} \put(3.1,1.85){\footnotesize \it $0$}

      \ncline[arrowsize=.2 3]{->}{2}{3}
      \ncline[arrowsize=.2 3]{->}{3}{4}
      \ncline[arrowsize=.2 3]{->}{2}{5}
      \ncline[arrowsize=.2 3]{->}{5}{4}
      \ncline[arrowsize=.2 3]{->}{4}{2}

 \cnode*(7,1.75){0.1}{l2} \put(7.1,1.85){\footnotesize \it $2$}
      \cnode*(8,0.){0.1}{l4} \put(8.,-0.4){\footnotesize \it $1$}
      \cnode*(6,0.){0.1}{l5} \put(5.6,-0.4){\footnotesize \it $0$}
       \cnode*(10,0.){0.1}{l1} \put(10.,-0.4){\footnotesize \it $0$}
  \cnode*(9,1.75){0.1}{l3} \put(9.1,1.85){\footnotesize \it $2$}

      \ncline[arrowsize=.2 3]{->}{l2}{l5}
      \ncline[arrowsize=.2 3]{->}{l5}{l4}
      \ncline[arrowsize=.2 3]{->}{l4}{l2}
       \ncline[arrowsize=.2 3]{->}{l1}{l4}
   \ncline[arrowsize=.2 3]{->}{l3}{l1}
       \ncline[arrowsize=.2 3]{->}{l4}{l3}
      \end{pspicture}
\caption{A digraph with a $(2,2)$-Grundy function and a digraph with $(3,2)$-Grundy function. \label{egrun}}
\end{center}
\end{figure}

\begin{remark}If a digraph $D$ has a $(k,l)$-Grundy function $g$, then $D$ has a $(k,l)$-kernel $K=\{ x\in V : g(x)=0\}$.
\end{remark}

In \cite{GPR:91} it was proved that  the number of Grundy functions of $D$ is equal to number of Grundy functions of  its line digraph. Next, we extend this result to $(k,l)$-Grundy functions and to partial line digraphs. First we prove that a digraph has a $(k,l)$-Grundy function if and only if any partial line digraph has.

\begin{lemma}\label{DLkl}Let $l\ge 1$ and $k\ge 2$ be two integers. Let $D$ be a digraph with minimum in-degree at least 1 having a $(k,l)$-Grundy function $g$.   Let $A'$ and $\phi$ satisfy the requirements of the definition of a partial line digraph, i.e., $\mathcal{L}_{(A',\phi)}D=\mathcal{L} D$.
Then $g_L:A' \to \mathbb{N}$  defined as $g_L(yx)=g(x)$ is a $(k,l)$-Grundy function on $\mathcal{L} D$.
\end{lemma}
\begin{proof}
 Let $g:V\to \mathbb{N}$ be a $(k,l)$-Grundy function on $D=(V,A)$. Next, we prove that  $g_L:A' \to \mathbb{N}$ be defined as $g_L(yx)=g(x)$ is a $(k,l)$-Grundy function on $\mathcal{L} D$.
Let $yx\in V(\mathcal{L} D)$. First, suppose that $g_L(yx)=t>0$. Since $g_L(yx)=g(x)=t>0$, by (1) of Definition \ref{grunkl}, it follows that for each $j<t$, there is  $w\in N_l^+(x)$ with $g(w)=j$. Hence, there is a path  $(x,x_1,\ldots, x_r=w)$ in $D$ with $r\le l$,  which produces a path  $(yx,\phi(xx_1),\ldots, \phi(x_{r-1}x_r))$ in $ \mathcal{L} D$ of length $r$, yielding that
$\phi(x_{r-1}w)\in N^+_l(yx)\subset V(\mathcal{L}D)$.
Therefore  for each $j<t$, there is  $\phi(x_{r-1}w)\in N^+_{l}(yx)\subset V(\mathcal{L}D)$ and   $g_L(\phi(x_{r-1}w))=g(w)=j$.  Thus, $g_L$ meets  requirement (1) of Definition \ref{grunkl}. Now suppose that $g_L(yx)=t$, so $g(x)=t$. Let  $uv \in N^+_{k-1}(yx)\subset V(\mathcal{L}D)$, then there is a
path  $(yx,\phi(xx_1),\ldots, \phi(x_{r-1}x_r)=uv)$  of length $r\le k-1$ in $ \mathcal{L} D$. Hence there is a path  $(x,x_1,\ldots, x_r=v)$ in $D$ with $r\le k-1$, yielding that $g(v)\ne t$ because $v\in N_{k-1}^+(x)$
applying (2) of Definition \ref{grunkl}.
Therefore for all $uv\in N^+_{k-1}(yx)$, we have $g_L(uv)=g_L(\phi(x_{r-1}x_r))=g(v)\ne t$. Thus, $g_L$ meets requirement (2) of Definition \ref{grunkl}, concluding that $g_L$ is a $(k,l)$-Grundy function on $\mathcal{L} D$.
\end{proof}

\begin{lemma}\label{LDkl}Let $l\ge 1$ and $k\ge 2$ be two integers such that $l\le k-1$. Let $D$ be a digraph with minimum in-degree at least 1. Let $A'$ and $\phi$ satisfy the requirements of the definition of a partial line digraph, i.e., $\mathcal{L}_{(A',\phi)}D=\mathcal{L} D$. Suppose that $g$ is a $(k,l)$-Grundy function on $\mathcal{L} D$.
 Then  $g_D:V \to \mathbb{N}$  defined as $g_D(x)=g(yx)$, $yx\in A'$, is a $(k,l)$-Grundy function on $D$.
\end{lemma}
\begin{proof}
 Let $g:A'\to \mathbb{N}$ be a $(k,l)$-Grundy function on  $\mathcal{L} D$.  First, let us prove that $g_D:V \to \mathbb{N}$ defined as $g_D(x)=g(yx)$ with $x\in V$ and $yx\in A'$ is a function. So assume that there are two  arcs $yx,y'x\in A'$,  such that $g(yx)\ne g(y'x)$. Suppose $0\le h=g(yx)< g(y'x)$,
  then there exists $uv\in N^+_l(y'x)\subset V(\mathcal{L} D)$ such that
  $g(uv)=h$ by condition (1) of Definition \ref{grunkl}. Then there is a
path  $(y'x,\phi(xx_1),\ldots, \phi(x_{r-1}x_r)=uv)$  of length $r\le l$ in $ \mathcal{L} D$, and also a path  $(yx,\phi(xx_1),\ldots, \phi(x_{r-1}x_r)=uv)$  of length $r\le l$ in $ \mathcal{L} D$ implying that $uv\in N^+_l(yx)\subseteq N^+_{k-1}(yx) $  because $l\le k-1$, and $g(yx)=g(uv)=h$
   which is a contradiction with (2) of Definition \ref{grunkl}. Therefore $g(yx)=g(y'x)$. Furthermore, for every $x\in V$, there is an arc $yx\in A'$ by definition of $\mathcal{L} D$. Hence, $g_D(x)$ exists for all $x\in V$. Thus $g_D$ is a function.

Next, we prove that $g_D$ is a $(k,l)$-Grundy function on $ D$.
Let $x\in V$. First, suppose that $g_D(x)=t>0$. Since $g_D(x)=g(wx)=t>0$ where $wx\in A'$, by (1) of Definition \ref{grunkl}, it follows that for each $j<t$, there is  $uv\in N_{l}^+(wx)\subset V(\mathcal{L}D)$ with $g(uv)=j$. Then there is  a
path  $(wx,\phi(xx_1),\ldots, \phi(x_{r-1}x_r)=uv)$  of length $r$ in $ \mathcal{L} D$, implying that $v\in N^+_l(x)\subset V(D)$ and
$g_D(v)=j$. Hence, $g_D$ satisfies (1) of Definition \ref{grunkl}.
Finally, suppose that $g_D(x)=t$, let us see that for all $y\in N^+_{k-1}(x)$, $g_D(y)\not=t$.
 We have $t=g_D(x)=g(wx)$ for  $wx\in A'$.
 Since for all $y\in N^+_{k-1}(x)$, there exists a path $(x,x_1,\ldots, x_r=y)$ of length $r\le k-1$ in $D$, it follows that
  $\phi(x_{r-1}y)\in N^+_l(wx)\subset V(\mathcal{L} D)$, yielding that $g(\phi(x_{r-1}y))\ne g(wx)=t$ because (2) of Definition \ref{grunkl}. As $g_D(y)=g(\phi(x_{r-1}y))$ it turns out that $g_D(y)\ne t$.
   Thus, $g_D$  meets requirement (2) of Definition \ref{grunkl}, and we conclude that $g_D$ is a $(k,l)$-Grundy function.
\end{proof}

As an immediate consequence of both Lemma \ref{DLkl} and Lemma \ref{LDkl}, we can write the following theorem.
\begin{theorem}
Let $l\ge 1$ and $k\ge 2$ be two integers with $l\le k-1$. A digraph $D$ with minimum in-degree at least 1 has a $(k,l)$-Grundy function if and only if  any partial line digraph
$\mathcal{L} D$ has a  $(k,l)$-Grundy function.
\end{theorem}

\begin{theorem}\label{number}
Let $l\ge 1$ and $k\ge 2$ be two integers with $l\le k-1$. Let $D$ be a digraph  with minimum in-degree at least 1. Then the number of $(k,l)$-Grundy functions of $D$ is equal to number of $(k,l)$Grundy functions of  any partial line digraph $\mathcal{L} D$.
\end{theorem}

\begin{proof}
 Let $A'$ and $\phi$ satisfy the requirements of the definition of a partial line digraph, i.e., $\mathcal{L}_{(A',\phi)}D=\mathcal{L} D$.
Denote by $\mathcal{F}$ the set of all $(k,l)$-Grundy functions on $D$, and by $\mathcal{F}^*$ the set of all $(k,l)$-Grundy functions on $\mathcal{L}D$.
 If $g\in \mathcal{F}$, then the function $g_L$ given by  Lemma \ref{DLkl}, belongs to $ \mathcal{F}^*$; and  if $h\in \mathcal{F}^*$, then the function $h_D$ given by Lemma \ref{LDkl}, belongs to $ \mathcal{F}$.

Let $f:\mathcal{F} \to \mathcal{F}^*$ be defined by $f(g)=g_L$. Let us prove that $f$ is an injective function.

 Let $g,g'\in \mathcal{F}$ be such that  $f(g)=f( g')$, that is $g_L=g'_L$. Let us show that $g=g'$. Since for all $x\in V$ there exists $yx\in A'$, and $g_L(yx)=g'_L(yx)$, it follows that $g_L(yx)=g(x)=g'(x)=g'_L(yx)$. Hence $g=g'$.
 Thus, $f$ is an injective function yielding that $|\mathcal{F}|\le |\mathcal{F}^*|$.

Let $f^*:\mathcal{F}^*\to \mathcal{F}$ be defined by $f^*(h)=h_D$. Let us prove that $f^*$ is an injective function.

Let $h,h'\in \mathcal{F^*}$ be such that  $f^*(h)=f^*( h')$, that is $h_D=h'_D$. Let us show that $h=h'$. Since for all $yx\in A'$ we have  $h(yx)=h_D(x)=h'_D(x)=h'(yx)$,  it follows that $h=h'$.
 Thus, $f^*$ is an injective function yielding that $|\mathcal{F^*}|\le |\mathcal{F}|$.

Hence we conclude that  $|\mathcal{F}|= |\mathcal{F}^*|$.
\end{proof}

\end{document}